\documentclass[a4paper]{amsproc}
\usepackage{amsmath,amssymb,graphicx}
\usepackage{amscd}

\newcommand{\cdummy}{\cdot}
\newcommand{\tmop}[1]{\ensuremath{\operatorname{#1}}}

\newcommand{\tmtextrm}[1]{{\rmfamily{#1}}}

\theoremstyle{plain}

\newtheorem{lemma}{Lemma}

\newtheorem{theorem}{Theorem}
\numberwithin{equation}{section}

\begin{document}

\title{Analytic solutions and Singularity formation for the Peakon b--Family
equations}

\author{G. M. Coclite         \and
        F. Gargano \and
	V. Sciacca%etc.
}
\address[G. M. Coclite ]{Department of Mathematics, University of Bari, via E. Orabona 4, 70125 Bari (Italy)}
\email[]{coclitegm@dm.uniba.it}
\address[F. Gargano ]{Department of Mathematics and Computer Science, University of Palermo, via Archirafi 34, 90123 Palermo (Italy)}
\email[]{gargano@math.unipa.it}
\address[V. Sciacca ]{Department of Mathematics and Computer Science, University of Palermo, via Archirafi 34, 90123 Palermo (Italy)}
\email[]{sciacca@math.unipa.it}

\maketitle

\begin{abstract}
Using the Abstract Cauchy-Kowalewski Theorem we prove that the $b$-family
equation admits, locally in
time, a unique analytic solution.
Moreover, if the initial data is real analytic and it belongs to $H^s$ with $s >
3/2$,
and the momentum density $u_0 - u_{0,{xx}}$ does
not change sign, we prove that the solution stays analytic globally in time, for
$b\geq 1$.
Using pseudospectral numerical methods, we study, also, the singularity
formation for the $b$-family equations
with the singularity tracking method. This method allows us to
follow the process of the singularity formation in the complex plane
as the singularity approaches the real axis, estimating the rate of decay of the
Fourier spectrum.

\end{abstract}

\section{Introduction}
\label{intro}
In 2003 Holm and Staley \cite{HS1,HS2} introduced, as a one-dimensional version
of active fluid transport, the {\em $b$-family equation}
\begin{equation}
\label{eq:0}
m_t+\underbrace{u m_x}_{\rm convection}+\underbrace{b m u_x}_{\rm
stretching}=0,\qquad  t > 0,\> x \in{\mathbb{R}},
\end{equation}
where $b$ is a real parameter, $u=u(x,t)$ the fluid velocity, and $m=u-u_{xx}$
the momentum density. 
The quadratic terms in \eqref{eq:0} represent the competition and balance, in
fluid convection 
between nonlinear transport and amplification due to $b$-dimensional stretching.
The parameter $b$ characterizes the equations because it is the ratio of
stretching to convective transport and also  the number of covariant dimensions
associated with 
the momentum density $m$.

Within this family of evolutionary equations we have both the Camassa-Holm
equation \cite{CH,J} in correspondence of $b$= 2 
and the Degasperis-Procesi equation \cite{DP} in correspondence of $b$ = 3 as
special cases. 
Using the method of asymptotic integrability, it has been shown that these are
the unique two cases in which 
\eqref{eq:0} is  completely integrable \cite{DP}.
A quite complete list of references on the Camassa-Holm, the
Degasperis-Procesi equations and the b--family equations can be found in \cite{CHK,CK06,JNZ12,Z10,ZC11,Z04}.

We can rewrite \eqref{eq:0} is the following third order form in terms only of
the unknown $u$
\begin{equation}
\label{eq:1}
u_t- u_{txx} + (b + 1)uu_x = bu_xu_{xx} + uu_{xxx},\qquad  t > 0,\> x
\in{\mathbb{R}}.
\end{equation}
Moreover, \eqref{eq:0} and \eqref{eq:1} are also equivalent to the following
hyperbolic elliptic system
\begin{equation}
\label{eq:1.1}
u_t + uu_x +P_x= 0,\qquad -P_{xx}+P= \frac{b}{2} u^2 + \frac{3 - b}{2} u_x^2.
\end{equation}
Since $e^{-|x|}/2$ is the Green's function of the Helmholtz operator
$1-\partial_x^2$ we can rewrite \eqref{eq:1.1} in the following
integro-differential form
\begin{equation*}
u_t + uu_x +\int_{\mathbb{R}} \frac{e^{-|x-y|}}{2}{\rm sign}(y-x)\left(
\frac{b}{2} u^2(y,t) + \frac{3 - b}{2} u_x^2(y,t)\right)dy=0.
\end{equation*}
 
 In this paper we are concerned with the initial value problem associated to
\eqref{eq:1}, therefore we augment the equation with  the initial condition
 \begin{equation}
 \label{eq:2}
 u(x,0)=u_0(x),\qquad x\in {\mathbb{R}}.
 \end{equation}
 Assuming that
 \begin{equation*}
 u_0\in H^s(\mathbb{R}), \>s>\frac32, \qquad \text{$u_0-u_{0,xx}$ has constant
sign}
 \end{equation*}
using the Kato's theory and giving a precise description of the blow-up
scenario, Gui, Liu, and Tian proved the existence of global in time solutions.

The rest of the paper is organized as follows. In Section \ref{functionalsett}
we introduce the functional spaces in which our results hold.
Section \ref{local} is devoted to the local in time existence of solutions for
the Cauchy problem associated to \eqref{eq:1}.
The existence of global in time analytical solution for \eqref{eq:1} is proved
in Section \ref{global}.
In Section \ref{singularity} we analyze  the singularity formation using the
singularity tracking method for different values of  $b>-1$.

\section{Local in time result}

\subsection{Functional settings and the ACK Theorem}
\label{functionalsett}

In this section we introduce the functional spaces where local in time
analyticity for the $b$-family equations will be proved (see
{\cite{LCS03}}{\cite{LSS05}}{\cite{SC98}}). We consider $x$ a complex variable, with $\Re x$
and $\Im x$ respectively 
its real and imaginary part. We consider $f (x)$ a complex function analytic in
\begin{equation}
  D (\rho) = \left\{ x \in \mathbb{C} \hspace{0.25em} : | \Im x| \leq \rho
  \right\} \label{strip}
\end{equation}
and $f \in L^2  (\Gamma (\Im x))$ for $\Im x \in (- \rho, \rho)$. The
path where the $L^2$ integration is performed is:
\begin{equation}
  \Gamma (b) = \left\{ x \in \mathbb{C} \hspace{0.25em} : | \Im x| = b
  \right\} . \label{path}
\end{equation}
This means that if $\Im x$ is inside $(- \rho, \rho)$, then $f (\Re x + i \Im
x)$ is a square integrable function of $\Re x$. We denote by $H^{0, \rho}$ the
set of complex functions $f (x)$ with the above properties and such that
\begin{equation}
|f|_{\rho} = \sup_{\Im x \in (- \rho, \rho)} \|f (\cdummy + i \Im x)\|_{L^2
   (\Gamma (\Im x))} < \infty 
\end{equation}
is satisfied.

We denote by $H^{k, \rho}$ the space of all complex functions $f (x)$ such
that
\begin{equation} 
\partial_x^j f \in H^{0, \rho}, \hspace{1em} {\rm for \hspace{1em} 0
   \leq j \leq k ; \hspace{2em} and \hspace{1em} \|f\|_{k, \rho} = \sum_{0
   \leq j \leq k} | \partial_x^i f|_{\rho} .} 
\end{equation}
In term of the Fourier transform $\hat{f} (\xi)$ of $f$, the previous
norm can be rewritten as
\begin{equation}
  \label{normH} \|f\|_{k, \rho} = \left( \int e^{2 \rho | \xi |} (1 + | \xi
  |^2)^k | \hat{f} (\xi) |^2 d \xi \right)^{1 / 2} .
\end{equation}
%The above expression of the norm is useful to define the space $H^{r, \rho}$
%with $r \in \mathbb{R}_+$, used in section 3.
We denote by $H^{1,\rho}_{\beta, T}$ the space of all functions $f (x, t) \in
C^1 ([0,
T], H^{1, \rho})$ and
\[ |f|_{1, \rho, \beta, T} = \sup_{0 \leq t \leq T} |f (\cdummy, t) |_{0, \rho
   - \beta t} + \sup_{0 \leq t \leq T} | \partial_t f (\cdummy, t) |_{0, \rho
   - \beta t} + \sup_{0 \leq t \leq T} | \partial_x f (\cdummy, t) |_{0, \rho
   - \beta t} \]
is finite. In an analogous way, $H^{k, \rho}_{\beta, T}$ is the space of all
functions \\$f \in C^k ([0, T], H^{k, \rho})$ and $ |f|_{k, \rho, \beta, T}  <
\infty $.

We conclude this section recalling the ACK theorem in the form given by Safonov
\cite{SAF95}. 

Consider the problem
\[ u = F (u, t), \]
where $u \in X_{\rho}$ and $X_{\rho}$ is a Banach space with norm denoted by
$| \cdummy |_{\rho}$.
\begin{theorem}\label{ACKt}
ACK Theorem: Suppose that $\exists R > 0$, $\rho_0 > 0$, and
$\beta_0 > 0$ such that if $0 < t \leq \rho_0 / \beta_0$, the following
properties hold:
\begin{enumerate}
  \item for all $0 < \rho' < \rho \leq \rho_0$ and $u \in X_{\rho}$ such that
  $\sup_{0 \leq t \leq T} |u (\cdummy, t) |_{\rho} \leq R$, the map $F (u, t)
  : [0, T] \mapsto X_{\rho'}$ is continuous;
  
  \item for all $0 < \rho < \rho_0$ the function
  \[ F (0, t) : [0, \rho_0 / \beta_0] \mapsto \{u \in X_{\rho} :  \sup_{0 \leq
     t \leq T} |u (\cdummy, t) |_{\rho} \leq R\} \]
  is continous and $|F (0, t) |_{\rho} \leq R_0 < R$;
  
  \item for all $0 < \rho' < \rho (s) \leq \rho_0$ and for all $w$ and $u \in
  \{u \in X_{\rho} : \\ \sup_{0 \leq t \leq T} |u (\cdummy, t) |_{\rho - \beta_0
  t} \leq R\}$, results
  \[ |F (u, t) - F (w, t) |_{\rho'} \leq C \int_0^t \frac{|u - w|_{\rho
     (s)}}{\rho (s) - \rho'} ds ; \]
\end{enumerate}
then there exists $\beta > \beta_0$ such that for all $0 < \rho < \rho_0$
equation () has a unique solution $u (\cdummy, t) \in X_{\rho}$ with $t \in
[0, (\rho_0 - \rho) / \beta_0]$ and moreover \\ $\sup_{\rho < \rho_0 - \beta t}
|u (\cdummy, t) |_{\rho} \leq R$.
\end{theorem}

\subsection{First main theorem: short time existence of analytic solution}
\label{local}
%\section{The ACK theorem for the b--family equation}

In this section we proof the following first main theorem: 

\begin{theorem}\label{primoteo} 
Local in time analyticity: Let the initial data for the $b$-family $u_0 \in
H^{1, \rho_0}$. Then there exists $\beta > 0$ such that for any $\rho$, with
$0 < \rho < \rho_0$, there exists a unique continuously differentiable w.r.t.
time solution $u$ of the $b$-family equation with the following property:
\begin{itemize}
  \item $u (\cdummy, t) \in H^{1, \rho}$ and $\partial_t u (\cdummy, t) \in
  H^{1, \rho}$, when $t \in \left[ 0, \frac{\rho_0 - \rho}{\beta} \right]$.
\end{itemize}
\end{theorem}

\begin{proof}

We apply the abstract Cauchy-Kovalevskaya (ACK) theorem to the
$b$-family equation to prove
the short time existence of analytic solution. We remand to \cite{LSS05} for the
details and here we give only a sketch of the proof.

To apply the ACK theorem to the $b$-family equation, we write the $b$-family
equation as:
%\[ (1 - \partial_x^2)  (u_t + uu_x) = (b - 3) u_x u_{xx} - buu_x = -
%   \partial_x  \left( \frac{b}{2} u^2 + \frac{3 - b}{2} u_x^2 \right), \]
%and finally in the pseudodifferential form
\begin{equation}
u_t + uu_x = - A^{- 2}  \left( \frac{b}{2} u^2 + \frac{3 - b}{2} u_x^2
\right)_x, \label{Beq1} 
\end{equation}
where $A^2 = (1 - \partial_x^2)$, and, with an integration in
time, we transform \eqref{Beq1} in
\begin{eqnarray}\label{ACKform}
  &  & u = F (u, t), \hspace{2em} \text{\tmtextrm{\tmop{with}}} \\
  \label{Beq} &  & F (u, t) = u_0 - \int_0^t \left[ uu_x + A^{- 2}  \left(
  \frac{b}{2} u^2 + \frac{3 - b}{2} u_x^2 \right)_x \right] dt' . \nonumber
\end{eqnarray}
It is clear that properties
1 and 2 of the ACK theorem are satisfied by the operator $F (u, t)$ given in
\eqref{ACKform}. For
the validity of hypothesis 3, one considers first the norm $|F (u, t) |_{1,
\rho'}$, where the bound is easily obtained using the following estimates
which are a direct consequence of the Cauchy estimate on the derivative of
analytic function (\cite{LCS03}{\cite{SC98}}):

\begin{lemma} Let $u, w \in H^{1, \rho''}$ and $\rho' < \rho''$, then
$ \|uw_x - wu_x \|_{1, \rho'} \leq c \frac{\|u - w\|_{1, \rho''}}{\rho'' -
   \rho'} $.
\end{lemma}
Finally, the bound on $| \partial_x F (u, t) |_{1, \rho'}$ is obtained
using the fact that the operator $A^{-2}$ is bounded. 
\end{proof}

\section{Global in time analyticity}
\label{global}

In this section we prove that the $b$-family equation admits a strong solution
in Gevrey-class functional space of index 1. This analysis gives an explicit
lower bound on the shrinking of the radius of analyticity of the solution of
the $b$-family equation with given analytical initial data, and assuming
suitable regularity in Sobolev space.

We consider the space $D (A^r e^{\rho A})$, which consists of the real
restriction of complex analytic functions on the strip $D (\rho)$. The 
operator $A^r e^{\rho A}$ is the Gevrey operator with $A^2 = (1 -
\partial_x^2)$.  
The space $D (A^r e^{\rho A})$ is equivalent to the space $H^{r,\rho}$ 
defined in section \ref{functionalsett}, with norm given by \eqref{normH}, then
we define 
the Gevrey class $G^1$ of index 1 by (see {\cite{LO97}}):
\begin{equation}
 G^1 = \bigcup_{\rho > 0} H^{r, \rho} .
\end{equation}
We suppose that the initial data $u_0 \in D (A^r e^{\rho_0 A})$, with 
$\rho_0$ and $r$ positive. 

In this section we prove the following second main theorem:

\begin{theorem}\label{seconteo} Global in time analyticity: Let $b \geq 1$ and
$A^2 = (1 -
\partial_x^2)$. Let
$u_0 \in D (A^r e^{\rho_0 A})$, with $r > 3 / 2$, $\rho_0 > 0$ and $m_0 = u_0
- u_{0 xx}$ does not change sign. Then the unique solution $u$ of the
$b$-family equation lies in Gevrey class of index 1 globally in time.
\end{theorem}

\begin{proof}
We start the
energy type estimate on $D (A^r e^{\rho A})$ for the $b$-family equation.
Supposing that $\rho$ is a decreasing, positive $C^1$ function of the time $t$
with $\rho (0) = \rho_0$, we consider the $L^2$-scalar product of the
$b$-family equation \eqref{Beq1} with $A^{2 r - 2} e^{2 \rho A} u$:
\begin{eqnarray}
  &  & \frac{1}{2}  \frac{d}{dt}  \|A^r e^{\rho A} u\|^2 - \dot{\rho}  \|A^{r
  + 1 / 2} e^{\rho A} u\|^2 = \nonumber\\
  &  & = | \langle A^r e^{\rho A} uu_x, A^r e^{\rho A} u \rangle + \langle
  A^{r - 1} e^{\rho A}  \left( \frac{b}{2} u^2 + \frac{3 - b}{2} u_x^2
  \right)_x, A^{r - 1} e^{\rho A} u \rangle | \nonumber\\
  &  & \leq | \langle A^r e^{\rho A} uu_x, A^r e^{\rho A} u \rangle | +
  \frac{|b|}{2} | \langle A^{r - 1} e^{\rho A} \left( u^2 \right)_x, A^{r - 1}
  e^{\rho A} u \rangle | \nonumber\\
  &  & \hspace{1em} + \frac{|b - 3|}{2} | \langle A^{r - 1} e^{\rho A} \left(
  u_x^2 \right)_x, A^{r - 1} e^{\rho A} u \rangle | .  \label{Zero}
\end{eqnarray}
Now, we start to estimate the three terms on the second member of
(\ref{Zero}),\\
$I_1=| \langle A^r e^{\rho A} uu_x, A^r e^{\rho A} u \rangle |$, 
$I_2=| \langle A^{r - 1} e^{\rho A} \left( u^2 \right)_x, A^{r - 1} e^{\rho
  A} u \rangle |$ and \\
$I_3=| \langle A^{r - 1} e^{\rho A} \left( u_x^2 \right)_x, A^{r - 1}
  e^{\rho A} u \rangle |$.\\
From Lemma \ref{lemma2} below, we have
\begin{equation}
  I_1 \leq C_1  \|A^r
  u\|^3 + C_2 \rho \|A^r e^{\rho A} u\|  \|A^{r + 1 / 2} e^{\rho A} u\|^2 .
\label{I1}
\end{equation}
Using Cauchy estimate and Lemma \ref{lemma1} below we obtain:
\begin{eqnarray}
  &  & I_2 \leq \langle A^r e^{\rho A} u^2, A^{r - 1} e^{\rho A} u
  \rangle | \leq \|A^r e^{\rho A} u^2 \|  \|A^{r - 1} e^{\rho A} u\| \\
\nonumber
  &  & \hspace{1em} \leq  \|A^r e^{\rho A} u\|^2  \|A^{r - 1} e^{\rho A} u\|
\leq \|A^r e^{\rho A} u\|^3 \\ \nonumber
  &  & \hspace{1em} 
  \leq C_1  \|A^r u\|^3 + \rho C_2  \|A^{r + 1 / 3} e^{\rho
  A} u\|^3 \\ \nonumber
  &  & \hspace{1em}
\leq C_1  \|A^r u\|^3 + \rho C_2  \left( \|A^r e^{\rho A}
  u\|^{1 / 3} \|A^{r + 1 / 2} e^{\rho A} u\|^{2 / 3} \right)^3 \\ \nonumber
  &  & \hspace{1em} \leq C_1  \|A^r u\|^3 + \rho C_2  \|A^r e^{\rho A} u\| 
  \|A^{r + 1 / 2} e^{\rho A} u\|^2 ; \label{I2}
\end{eqnarray}
\begin{eqnarray}
  &  & I_3 \leq \langle A^{r - 1} e^{\rho A} u_x^2, A^r e^{\rho A} u
  \rangle | \leq \|A^{r - 1} e^{\rho A} u_x^2 \|  \|A^r e^{\rho A}
  u\| \\ \nonumber
  &  & \hspace{1em} \leq \|A^{r - 1} e^{\rho A} u_x \|^2  \|A^r e^{\rho A}
  u\| \leq \|A^r e^{\rho A} u\|^3 \\ \nonumber
  &  & \hspace{1em} \leq C_1  \|A^r u\|^3 + \rho C_2  \|A^{r + 1 / 3} e^{\rho
  A} u\|^3  \\ \nonumber
  &  & \hspace{1em} \leq C_1  \|A^r u\|^3 + \rho C_2  \left( \|A^r e^{\rho A}
  u\|^{1 / 3} \|A^{r + 1 / 2} e^{\rho A} u\|^{2 / 3} \right)^3 \\ \nonumber
  &  & \hspace{1em} \leq C_1  \|A^r u\|^3 + \rho C_2  \|A^r e^{\rho A} u\| 
  \|A^{r + 1 / 2} e^{\rho A} u\|^2 .\label{I3}
\end{eqnarray}
Collecting all the previous estimates \eqref{I1},\eqref{I2} and \eqref{I3}, we
obtain
\begin{equation} 
\frac{1}{2}  \frac{d}{dt}  \|A^r e^{\rho A} u\|^2 \leq \left( \dot{\rho} +
   C_2 \rho \|A^r e^{\rho A} u\| \right)  \|A^{r + 1 / 2} e^{\rho A} u\|^2 +
   C_1  \|A^r u\|^3 , \label{lastineq} 
\end{equation}
where constants $C_1$ and $C_2$ depend only on $b$ and $\rho$. The last
inequality \eqref{lastineq} implies that $u \in D (A^r e^{\rho A})$ with
\begin{equation}
 \rho (t) = \rho_0 e^{- C_2  \int_0^t \left( \|A^r e^{\rho_0 A} u_0 \|+ C_1
   \left( \int_0^{t'} \|A^r u\|^3 dt^" \right) \right) dt'} . \label{least}
\end{equation}

To prove the global in time analyticity of the solution $u$, it is necessary
to have a priori Sobolev estimate on $\|A^r u\|$. To do that we use the global
existence result for the $b$-family equation, given by the result of Lemma \eqref{lemma3} below. 
\end{proof}

\begin{lemma}\label{lemma1} 
Let $\rho \geq 0$ and let $u, v \in D (A^r e^{\rho
A})$ with $r > 1 / 2$; and let $w \in D (A^{r + s} e^{\rho A})$, with $\rho,
r, s \geq 0$. There exists constants $C, C_1$ and $C_2$, depending only on
$r$, such that
\[ \|A^r e^{\rho A} (uv)\| \leq C \|A^r e^{\rho A} u\|  \|A^r e^{\rho A} v\| ;
\]
\[ \|A^r e^{\rho A} w\| \leq C_1  \|A^r w\| + \rho C_2  \|A^{r + s} e^{\rho A}
   w\| . \]
where $\| \cdot \|$ is the usual $L^2$ norm.
\end{lemma}
\begin{proof}
See {\cite{LO97}}. 
\end{proof}
\begin{lemma}\label{lemma2} 
Let $v \in D (A^r e^{\rho A})$ and $u \in D (A^{r + 1
/ 2} e^{\rho A})$, with $\rho \geq 0$ and $r > 3 / 2$, then there exists
constants $C_1$ and $C_2$, depending only on $r$, such that
\[ | \langle A^r e^{\rho A} vu_x, A^r e^{\rho A} u \rangle | \leq C_1  \|A^r
   v\|  \|A^r u\|^2 + C_2 \rho \|A^r e^{\rho A} v\|  \|A^{r + 1 / 2} e^{\rho
   A} u\|^2 . \]
\end{lemma}
\begin{proof} 
See {\cite{BG03}} and {\cite{LSS05}}.
\end{proof}

\begin{lemma}\label{lemma3} 
Let $b \geq 1$ and let $u_0 \in H^s (
\mathbb{R}) \cap L^1 ( \mathbb{R})$, with $s > 3 / 2$ and $m_0 = u_0 - u_{0
xx}$ does not change sign, then the corresponding solution does exists
globally. 
\end{lemma}
\begin{proof}
See {\cite{GLT}} for the proof. 
\end{proof}

\section{Singularity formation for the $b$-family equation}
\label{singularity}
In this section we present the analysis of the singularity formation for the
$b$-family equation using the singularity tracking method for different values
of the parameter $b$. The singularity tracking method is based on the following
Laplace asymptotic formula (see {\cite{CKP66}}) for the Fourier spectrum of a
complex function $u (z)$:
\begin{equation}
  \label{laplace} \hat{u}_k \sim C |k|^{- (1 + \alpha)} \exp (- \delta |k|)
  \exp (i|k|x^{\star}),
\end{equation}
where $z^{\star} =x^{\star} + i \delta$ is the nearest complex singularity to
the real axis,
$\alpha$ is the algebraic type, and at the singularity $u (z) \sim (z -
z^{\star})^{\alpha}$.
Therefore an estimation of $\delta$ and $x^*$ gives the complex location of the
singularity.
Resolving the rate of algebraic decay $1+\alpha$, one can classify the
singularity type.
In particular, when one deals with a PDE with solution $u(t)$,
if at a critic time $t_s$ results that $\delta (t_s) = 0$ then the
width of analyticity of the solution is zero and one can observe a blow up of
the solution or a blow up of its derivatives. The reader is invited to
see {\cite{C93,DLSS06,FMB03,GSS09,PS98,S92,SSF83}} for more details on this
method.

To apply the singularity tracking method to the $b$-family, we numerically solve
the equations
in the spatial domain $[-\pi,\pi]$, employing periodic boundary conditions. 
We write the $b$-family equations in pseudodifferential
form
\[ u_t + uu_x = - A^{- 2}  \left( \frac{b}{2} u^2 + \frac{3 - b}{2} u_x^2
   \right)_x , \]
and the dynamic of the $k$th Fourier mode of $u$ is described by the following
ODE
\begin{equation}
  \label{numode} \partial_t  \hat{u}_k = - \left[ \widehat{(uu_x)}_k +
  \frac{ik}{1 + k^2}  \left( \frac{b}{2} \widehat{(u^2)}_k + \frac{3 - b}{2}
  \widehat{(u_x^2)}_k \right) \right] .
\end{equation}
Dividing the time interval $[0, T]$ in N sub--intervals of size $\Delta t = T
/ N$, we approximate the solution $u$ as
\[ u (x, n \Delta t) \approx \sum_{k = - K / 2}^{K / 2} \hat{u}_k^n e^{ikx} ;
\]
and we solve the system of ODEs (\ref{numode}) using explicit Runge--Kutta
method of the 4th order with initial conditions given by
\[ \hat{u}_k^0 = \hat{u}_{0_k}, \]
where $\hat{u}_{0_k}$ are the Fourier coefficients of the initial data $u (x, 0)
=
u_0 (x)$.
The reader is referred to {\cite{DLSS06}} for the analysis of the convergence
properties 
of this numerical scheme for the Camassa--Holm equation case, which is the
$b$-family equation with
 $b= 2$.

The estimation of the parameters $C$, $\alpha$ and $\delta$ in the
Laplace asymptotic formula (\ref{laplace}) is performed with
the technique of the sliding--fitting  with length $3$ (see
{\cite{C93,DLSS06,PMFB06,PS98,S92}} for details). This
procedure is based on the local estimation of the wanted values for each $k$
mode, by the
formulas:
\begin{eqnarray}
  \alpha (k) & = & \frac{\log \left( \frac{\hat{u}_{k - 1}  \hat{u}_{k +
  1}}{\hat{u}^2_k} \right)}{\log \left( \frac{k^2}{\left( k - 1 \right) 
  \left( k + 1 \right)} \right)}, \\
  \delta (k) & = & \left[ \log \left( \frac{\hat{u}_k}{\hat{u}_{k + 1}}
  \right) + \alpha \log \left( \frac{k}{k + 1} \right) \right], \\
  \log C (k) & = & \log | \hat{u}_k | + \alpha (k) \log \left( k \right) + k
  \delta (k) . 
\end{eqnarray}
The asymptotic values  for large $k$ can be computed with an
extrapolation process using the epsilon algorithm of Wynn (see
{\cite{PMFB06}}), giving the wanted values $C$, $\alpha$ and $\delta$ .

For all the numerical simulations we shall present in this work, we have
 used up to 4096 Fourier modes calculated with 32-digit(using the ARPREC
package\cite{BY}).

\subsection{Numerical solutions}
\label{numerical}
For the numerical results we consider two different initial data: the first
one is $u_0 = \sin (x)$ and the second one is $u_0 = 1 + \sin (x)$, and in the
sequel we call them type I e II respectively. These initial conditions does not
 satisfy the hypothesis of theorem 2, therefore the
 solutions are candidate to develop a singularity in a finite
time. This work want to extend the results found in {\cite{DLSS06}}, where the
singularities formation for Camassa--Holm equation was analyzed. In
particular, for initial condition of type I, Camassa--Holm equation develops a
singularity in a finite time of algebraic type 3/5, weaker with respect to the
cubic--root singularity developed by Burgers equation {\cite{SSF83}}, probably
due to the presence of the dispersive terms. For initial condition of type II,
Camassa--Holm equation develops in a finite time singularity of algebraic type
2/3, which is the behavior of the peakons solution for Camassa--Holm equation.

Here we present the numerical results of the
Degasperi--Procesi equation, which corresponds to the $b$-family equation with
$b= 3$ with respect
to initial data of type I and II.

\begin{figure}[ht]
  \resizebox{11cm}{5cm}{\includegraphics{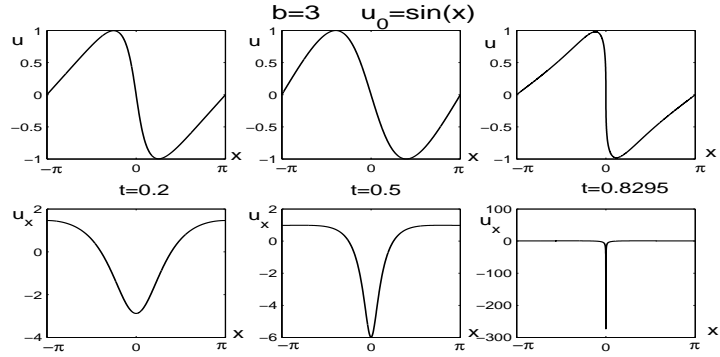}}
  \caption{The behavior in time of the numerical solution of
  Degasperi--Procesi equation with initial condition of type I. At time $t_s  =
  0.8295$ a singularity forms as a blow up of the first derivative.}
  \label{ub3sin}
\end{figure}

\begin{figure}[ht]
  \resizebox{11cm}{5cm}{\includegraphics{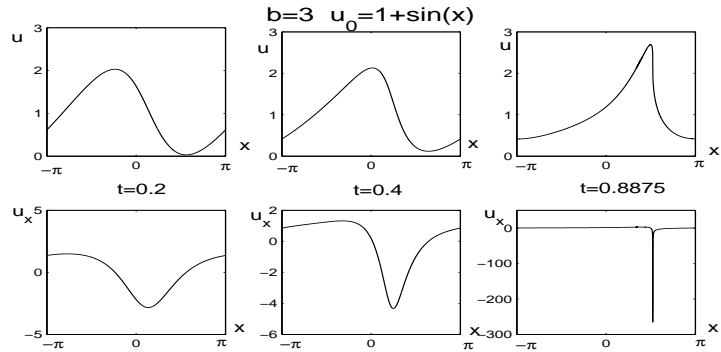}}
  \caption{The behavior in time of the numerical solution of
  Degasperi--Procesi equation with initial condition of type II. At time $t_s 
  = 0.8875$ a singularity forms as a blow up of the first derivative.}
  \label{ub31psin}
\end{figure}

The behaviors in time of the numerical solutions and their derivatives of the
Degasperi--Procesi equation with initial condition  I and II, are shown
respectively in Fig.\ref{ub3sin} and Fig.\ref{ub31psin}. The critical times
when singularities develop are respectively $t_s = 0.8295$ and $t_s =
0.8875$. At these critical times, it is evident a blow up of the first
derivatives of the solutions.

The analysis of the singularities formation and their characterization is
remanded in the sequel, where we perform a complete analysis for all the
$b$-family equations.

\section{Classifications of complex singularities: initial data $u_0 = \sin
(x)$.}
\label{numerical_I}
In this section we report the results obtained by applying the singularity
tracking method
to the $b-$family equation with initial datum of type I.

In Fig.\ref{spettrob3sin} it is shown the behavior in time of the
spectrum  of
the Degasperi--Procesi numerical solution ($b=3$),
starting at time $t = 0.35$ up to singularity time $t_s = 0.8295$ with
increments of $0.05$: at critical time $t_s$ the spectrum totally loss
exponential decay, as a singularity  of
cubic--root type hits the real axis, with the blow--up of the first derivative
of the solution.
The extrapolation of $\delta(t)$ and $\alpha(t)$ (on 
Fig.\ref{fittb3sin}) is obtained by the sliding 
fitting procedure of length $3$

\begin{figure}[ht]
  \resizebox{11cm}{5cm}{\includegraphics{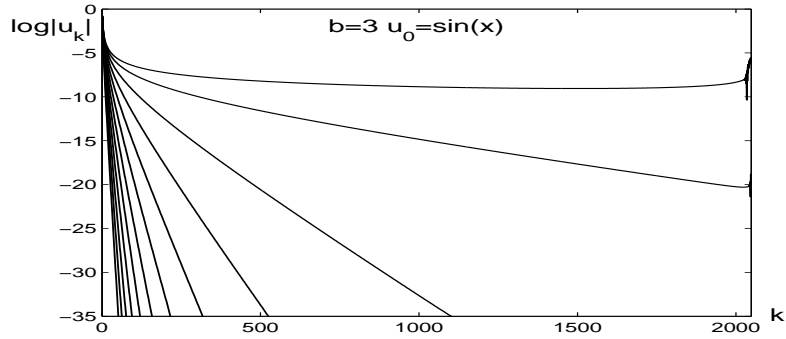}}
  \caption{The behavior of the spectrum in time of the
Degasperi--Procesi
  equation numerical solution with initial datum of type I, starting at time
  $t = 0.35$ up to singularity time $t_s = 0.8295$ with increments of $0.05$.}
  \label{spettrob3sin}
\end{figure}

\begin{figure}[ht]
  \resizebox{11cm}{5cm}{\includegraphics{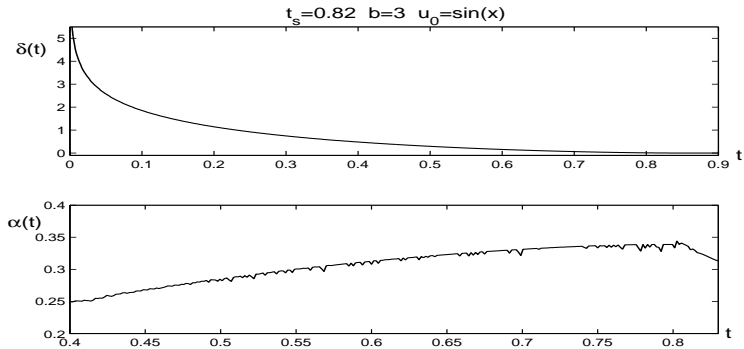}}
  \caption{The tracking singularity method results of the
Degasperi--Procesi
  equation numerical solution with initial datum of type I.
On the top,
  the evolution in time of the width of the analyticity strip $\delta$. On the
  bottom, the evolution in time of the algebraic character $\alpha$. 
The algebraic character at singularity time is $\alpha(t_s)=0.33$.}
  \label{fittb3sin}
\end{figure}

We have also performed several other numerical simulations for the initial
condition I, 
by varying the parameter $b>-1$. We have obtained that in each case 
a singularity develops in a finite time, and in Fig.\ref{tempisin} and Fig.\ref{alphasin} the
algebraic
character of the singularities and the times of singularity formation are shown
with respect to the
parameter $b$. We remember that for $b=-1$ the solution for the initial datum
of type I is
a traveling wave solution, and therefore in this case the solution is globally
analytic in time.

\begin{figure}[ht]
  \resizebox{11cm}{5cm}{\includegraphics{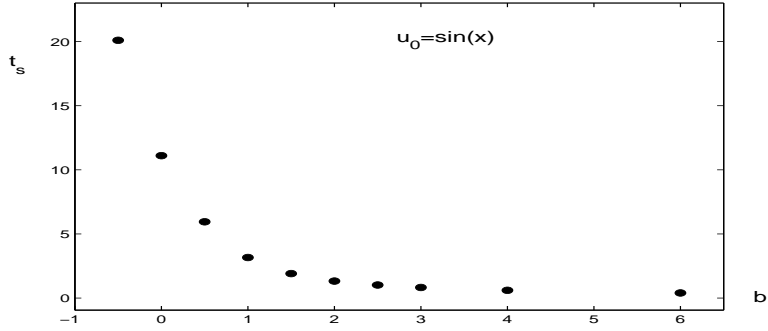}}

  \caption{The times of singularity for the $b$-family equation
with initial
  condition of type I with respect to the parameter $b$. On the right, the
singularity 
  algebraic characters $\alpha$ for the $b$-family
  equation with initial condition of type I with respect to the parameter
  $b$.}
  \label{tempisin}
\end{figure}

\begin{figure}[ht]
  \resizebox{11cm}{5cm}{\includegraphics{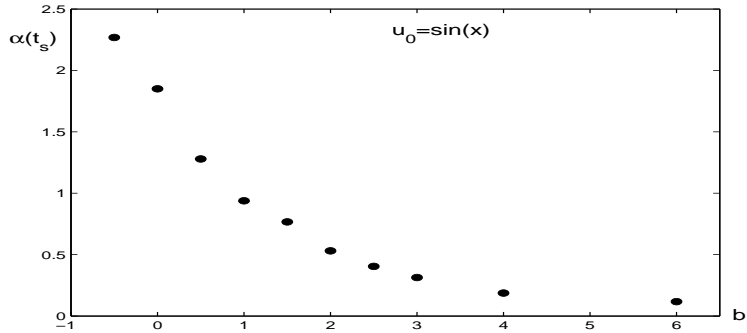}}

  \caption{The
singularity 
  algebraic characters $\alpha$ for the $b$-family
  equation with initial condition of type I with respect to the parameter
  $b$.}
  \label{alphasin}
\end{figure}

\section{Classifications of complex singularities: initial data $u_0 = 1 +
\sin (x)$.}
\label{numerical_II}
In this section we report the results obtained by applying the singularity
tracking method
to the $b$-family equation with initial datum of type II.

In Fig.\ref{spettrob31psin}  it is shown the behavior in time
of the spectrum 
of of the Degasperi--Procesi numerical solution,
starting at time $t = 0.4$ up to singularity time $t_s = 0.8875$ with
increments of $0.05$. On Fig.\ref{fittb31psin}
the evolutions in time of the width of the
analyticity strip $\delta$ and the algebraic character $\alpha$ are shown. 
One can see that at the singularity time $t_s = 0.8875$ the algebraic
character is $\alpha (t_s) = 0.41$, meaning that a blow--up of the first
derivative of the solution occurs.
The results are also obtained by applying the sliding fitting procedure with length $3$.

\begin{figure}[ht]
  \resizebox{11cm}{5cm}{\includegraphics{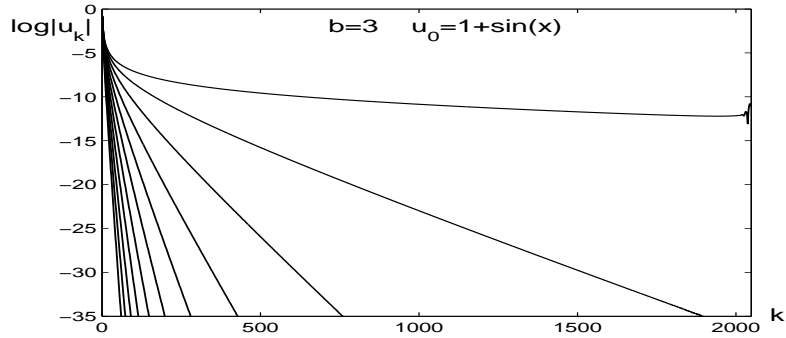}}
  \caption{The behavior in time of the spectrum  of the
Degasperi--Procesi
  equation numerical solution with initial datum of type II,starting at time
  $t = 0.4$ up to singularity time $t_s = 0.8875$ with increments of $0.05$.}
  \label{spettrob31psin}
\end{figure}

\begin{figure}[ht]
  \resizebox{11cm}{5cm}{\includegraphics{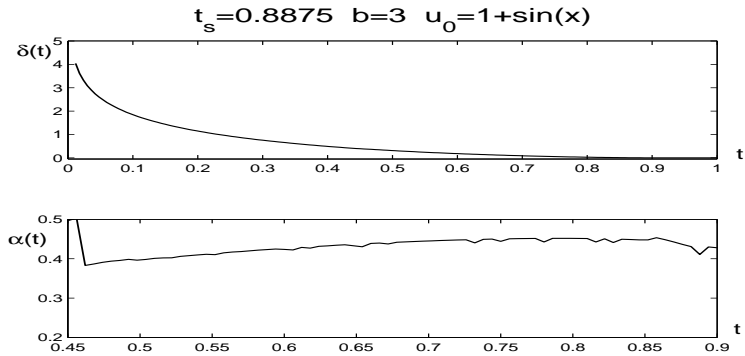}}
  \caption{The tracking singularity method results of the
Degasperi--Procesi
  equation numerical solution with initial datum of type II up to singularity time $t_s = 0.8875$. 
On the top, the evolution in time of the width of the analyticity strip $\delta$. On the
  bottom, the evolution in time of the algebraic character $\alpha$. 
The algebraic character at singularity time is $\alpha(t_s)=0.41$.
}
  \label{fittb31psin}
\end{figure}

Also for the initial condition of type II we have  performed several other
numerical simulations  
by varying the parameter $b$. The singularity time and the algebraic
character of the singularities are shown in Fig.\ref{tempi1psin} and Fig.\ref{alpha1psin}
w.r.t the
parameter $b$. As observed in the case of the initial condition of type I, for
$b= - 1$ the solution 
for the initial data of type II is a traveling wave solution, and therefore the
solution is globally analytic in time.

\begin{figure}[ht]
  \resizebox{11cm}{5cm}{\includegraphics{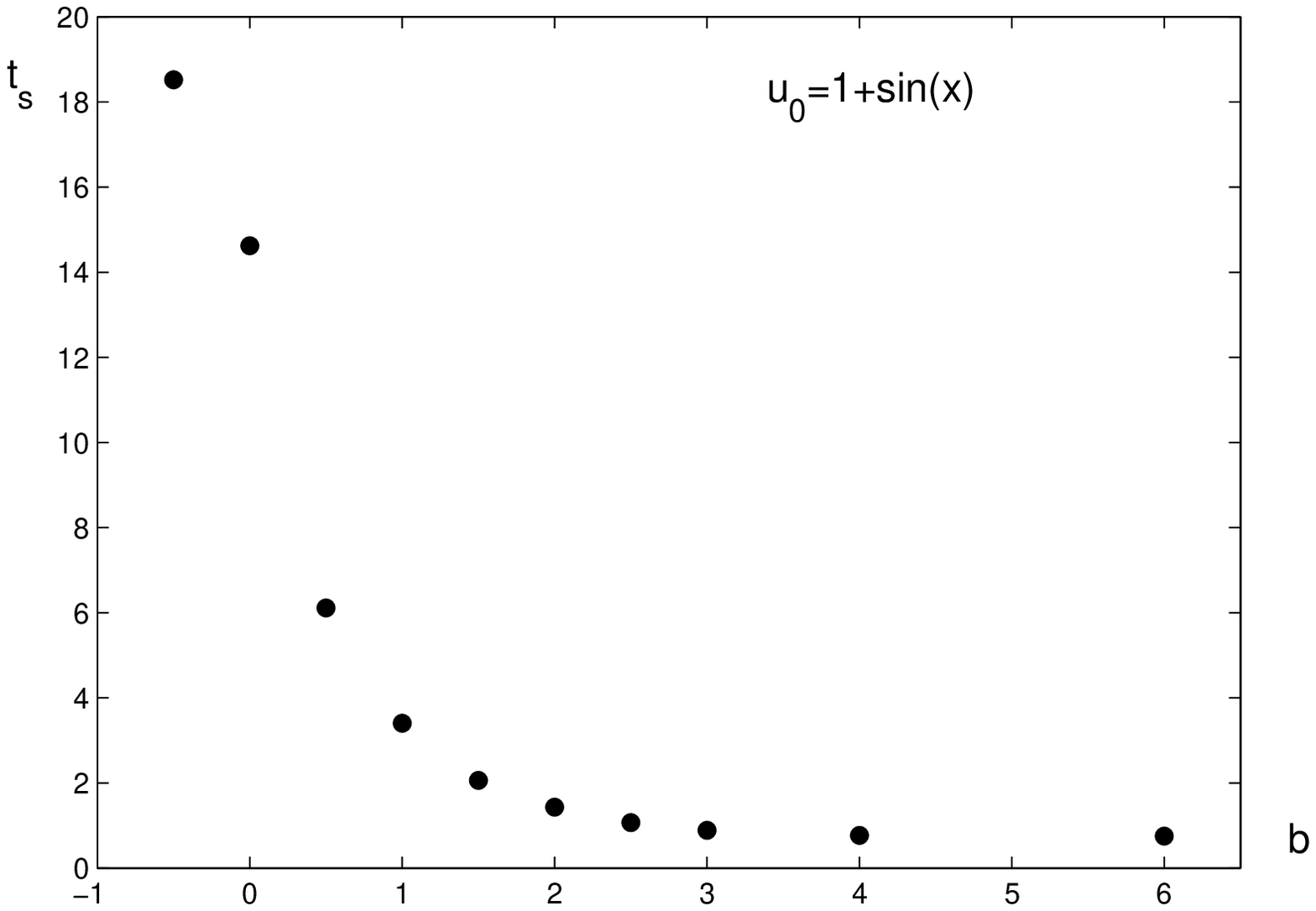}}

  \caption{The times of singularity for the $b$-family equation
with initial
  condition of type II with respect to the parameter b.}
  \label{tempi1psin}
\end{figure}

\begin{figure}[ht]
  \resizebox{11cm}{5cm}{\includegraphics{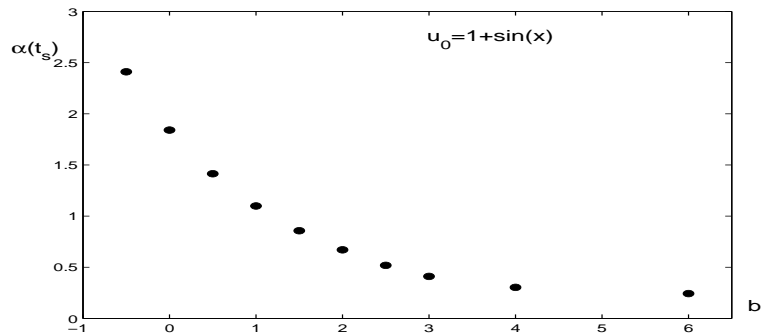}}

  \caption{The
singularity algebraic characters $\alpha$ for the $b$-family
  equation with initial condition of type II with respect to the parameter
  b.}
  \label{alpha1psin}
\end{figure}

\section{Conclusions}
\label{Conclusion}
In this paper we have considered the $b$-family equations, introduced by Holm
and Staley \cite{HS1,HS2}.
Our main interest has been to show the well-posedness of the $b$-family equation
in analytic
function spaces. Using the Abstract Cauchy-Kowalewski Theorem, we prove that the
$b$-family
equation admits a unique analytic solution, 
both locally in time (Theorem \eqref{primoteo}) then globally in time (Theorem
\eqref{seconteo}).
An interesting problem in the theory of PDE’s is the singularity
formation and how to detect it from a numerical point of view. One of the most
effective method to achieve this goal is the singularity tracking method. It
consists to follow the singularity in the complex
plane before the appearance as a blow up of the solution. The distance of the
singularity from the
real axis can be estimated through the study of the exponential rate of decay of
the Fourier spectrum. 
The algebraic character of the singularity is determinate, instead, by the the
study of the algebraic rate of decay of the Fourier spectrum. We use numerical
pseudo–spectral methods
to study the process of singularity formation for the $b$-family equations,
w.r.t the
parameter $b>-1$, and the singularity time and the algebraic
characters are shown in Sections \ref{numerical_I} and \ref{numerical_II}.
We analysed also the cases for $b<-1$ and also in these cases the solutions
develop a singularity, but it was complicated to determinate their algebraic
character. Very high numeric precision is needed and it will be the subject of
future work.

%\begin{acknowledgements}
%\end{acknowledgements}


\begin{thebibliography}{}
% etc

\bibitem{BG03} 
J.L. Bona, Z. Gruji\'c, Spatial Analyticity Properties of
 Nonlinear Waves, Math. Models and Meth. in Appl. Sci., 13(3), 345--360 (2003).

\bibitem{BY}
D. H. Bailey, H. Yozo, X.S. Li, B. Thompson,
ARPREC, Lawrence Berkeley National Lab. (2002). 
  
\bibitem{C93} 
R.E. Caflisch, Singularity formation for complex solutions
  of the 3D incompressible Euler equations, Physica D,
67, 1--18 (1993).
  
\bibitem{CH}  
  R. Camassa, D. Holm, An integrable shallow water equation with peaked 
solitons, Phys. Rev. Letters, 71, 1661--1664 (1993). 
  
\bibitem{CKP66}
 G.F. Carrier, M. Krook, C.E. Pearson, Function of a
  Complex Variable: Theory and Technique, McGraw--Hill, New York (1966).
 
 \bibitem{CHK}
G.M.  Coclite, H. Holden, K.H. Karlsen,  Global weak solutions to a generalized
hyperelastic-rod wave equation, SIAM J. Math. Anal. 37, 1044--1069 (2005).
  
\bibitem{CK06} 
G.M. Coclite, K.H. Karlsen, On the well--posedness of the
  Degasperis--Procesi equation, J. Funct. Analys., 233, 60--91 (2006).
  
%\bibitem{CK07}
% G.M. Coclite, K.H. Karlsen, On the uniqueness of
%  discontinous solutions to the Degasperis--Procesi equation, J.
%  Diff. Equations, 234, 142--160 (2007).
  
\bibitem{DP}
  A. Degasperis, M. Procesi, Asymptotic integrability, in Symmetry and
Perturbation Theory, edited by A. Degasperis and G. Gaeta, World Scientific
(1999), 
23--37. 

 \bibitem{DLSS06} 
G. Della Rocca, M.C. Lombardo, M. Sammartino, V.
  Sciacca, Singularity tracking for Camassa-Holm and Prandtl's equations,
  Applied Numer. Math., 56, 1108--1122 (2006).
  
\bibitem{GSS09} 
F. Gargano, M. Sammartino, V. Sciacca, Singularity
  formation for Prandtl's equations, Physica D, 238, 1975--1991 (2009).  
  
\bibitem{GLT} 
G. Gui, Y. Liu, L. Tian, Global Existence and Blow--Up
  Phenomena for the Peakon b--Family of Equations,  Indiana Univ. Math. J. 57,
1209--1234 (2008).
  
\bibitem{FMB03} 
U. Frisch, T. Matsumoto, J. Bec, Singularity for Euler
  flow? Not out of the blue!, J. Statist. Phys., 113, 761--781 (2003).
  
\bibitem{HS1}   
  D.D. Holm, M.F. Staley, Wave structure and nonlinear balances in 
a family of 1+1 evolutionary PDEs, SIAM J. Appl. Dyn. Syst. (electronic), 2,
323--380 (2003).

\bibitem{HS2}  
 D.D. Holm, M.F. Staley, Nonlinear balances and exchange of stablity 
in dynamics of solitons, peakons, ramp/cliffs and leftons in a 1 + 1 nonlinear 
evolutionary PDE, Phy. Lett. A, 308, 437--444 (2003). 

\bibitem{J}  
   R.S. Johnson, Camassa-Holm, Korteweg-deVries and related models for water
waves, J. Fluid. Mech. 455, 63--82(2002). 
  
\bibitem{LO97} 
C.D. Levermore, M. Oliver, Analyticity of Solution for a
  Generalized Euler Equation, J. Differential Equation,
133, 321--339 (1997).

\bibitem{LCS03}
M.C. Lombardo, M. Cannone, M. Sammartino, Well-posedness of the boundary layer equations, 
SIAM J. Math. Anal., 35(4), 987--1004 (2003).
  
\bibitem{LSS05} 
M.C. Lombardo, M. Sammartino, V. Sciacca, A note on the
  analytic solutions of the Camassa-Holm equation, C. R. Math. Acad.
  Sci. Paris, 341 (11), 659--664 (2005).
  
\bibitem{PMFB06} 
W. Pauls, T. Matsumoto, U. Frisch, J. Bec, Nature of
  complex singularities for the 2D Euler equation, Physica D,
219, 40--59 (2006).
  
\bibitem{PS98} 
M.C. Pugh, M.J. Shelley, Singularity formation in thin jets
  with surface tension, Comm. Pure Appl. Math., 51(7),
733--795  (1998).
  
\bibitem{SAF95} 
M.V. Safonov, An Abtract Cauchy--Kovalevskaya Theorem in
  Weighted Banach Space, Comm. Pure Appl. Math., XLVIII, 629--637 (1995).
  
\bibitem{SC98} 
M. Sammartino, R.E. Caflisch, Zero viscosity limit for
  analytic solutions of the Navier-Stokes equation on a half-space, Comm. Math. Phys., 192(2), 433--461
(1998).
  
\bibitem{S92} 
M.J. Shelley, A study of singularity formation in
  vortex--sheet motion by spectral accurate vortex method, J. Fluid
  Mech., 244, 493--526 (1992).
  
\bibitem{SSF83} 
C. Sulem, P.--L. Sulem, H. Frisch, Tracing complex
  singularities with spectral methods, J. Comput. Phys.,
50, 867--877 (1983).

\bibitem{JNZ12}
Jiang, Zaihong; Ni, Lidiao; Zhou, Yong; Wave Breaking of the Camassa–Holm Equation, To appear in J. Nonlinear Sci.  DOI 10.1007/s00332-011-9115-0 (2012).

\bibitem{Z10}
Zhou, Yong; On solutions to the Holm-Staley b-family of equations, Nonlinearity 23 (2010), no. 2, 369–381.

\bibitem{ZC11}
Zhou, Yong; Chen, Huiping; Wave breaking and propagation speed for the Camassa-Holm equation with $k \neq 0$, Nonlinear Anal. Real World Appl. 12 no. 3, 1875–1882 (2011).

\bibitem{Z04}
Zhou, Yong; Blow-up phenomenon for the integrable Degasperis-Procesi equation, Phys. Lett. A 328, no. 2-3, 157–162 (2004).
\end{thebibliography}


\end{document}